\documentclass[12pt,oneside]{amsart}
\usepackage{geometry}                
\usepackage{graphicx}
\usepackage{amssymb}
\usepackage{enumitem}
\usepackage{epstopdf}
\usepackage{color}
\usepackage{caption}
\usepackage{mathrsfs}
\usepackage[mathscr]{euscript}
\title{On the discrepancies  of graphs}

\author[Balogh]{J\'ozsef Balogh}\address{Department of Mathematical Sciences, University of Illinois at Urbana-Champaign, Illinois, USA, 
and Moscow Institute of Physics and Technology,  Russian Federation.
}

\thanks{The first author's research is partially supported by NSF
Grant DMS-1500121 and DMS-1764123, Arnold O. Beckman Research Award (UIUC Campus Research Board RB 18132) and
the Langan Scholar Fund (UIUC)}

\email{jobal@illinois.edu}

\author[Csaba]{B\'ela Csaba}\address{Bolyai Institute, Interdisciplinary Excellence Centre, University of
Szeged, Hungary}
\email{bcsaba@math.u-szeged.hu}

\thanks{The second and fourth authors were partially supported by the NKFIH grants KH\_18 129597 and SNN 117879. 
The second author was also partially supported by  the Ministry of Human Capacities, Hungary, Grant 20391-
3/2018/FEKUSTRAT}

\author{Yifan Jing}
\address{Department of Mathematics, University of Illinois at Urbana-Champaign,  Illinois, USA}
\email{yifanjing17@gmail.com}

\author[Pluh\'ar]{Andr\'as Pluh\'ar}\address{Department of Computer
Science, University of
Szeged, Hungary}
\email{pluhar@inf.u-szeged.hu}

\begin{document}
\maketitle

\newtheorem{theorem}{Theorem}[section]
\newtheorem{lemma}[theorem]{Lemma}
\newtheorem{proposition}[theorem]{Proposition}
\newtheorem{corollary}[theorem]{Corollary}
\newtheorem{observation}[theorem]{Observation}
\newtheorem{conjecture}[theorem]{Conjecture}
\theoremstyle{definition}
\newtheorem{definition}[theorem]{Definition}
\newtheorem{example}[theorem]{Example}
\newtheorem{remark}{Remark}
\def\T{\mathfrak{T}}
\def\PP{\mathfrak{P}}
\def\D{\mathscr{D}}
\def\C{H}
\def\B{\mathcal{B}}
\def\ng{\widetilde{\mathsf{g}}}
\def\stime{\,\square\,}
\newcommand{\Case}[2]{\noindent {\bf Case #1:} \emph{#2}}
\newcommand{\Subcase}[2]{\noindent {\bf Subcase #1:} \emph{#2}}

\begin{abstract}
In the literature, the notion of discrepancy is used in several contexts, even in the theory of graphs. Here, for
a graph $G$, $\{-1, 1\}$ labels are assigned to the edges, and we consider a  family $\mathcal{S}_G$ of (spanning) subgraphs of certain types, among others 
 spanning trees, Hamiltonian cycles. As usual, 
we seek for bounds on the sum of the labels that hold for all elements of $\mathcal{S}_G$, for every labeling.    
\end{abstract}



\section{Introduction}
\label{sec:intro}
\thispagestyle{empty}

The thorough study of discrepancy theory started with Weyl \cite{Weyl} and quickly gained
several applications in number theory, combinatorics, ergodic theory, discrete geometry, 
statistics etc, see the monograph of Beck and Chen \cite{BeckChen} or the book chapter
by Alexander and Beck \cite{ABC}.

We touch upon only the combinatorial discrepancy of hypergraphs. Given a hypergraph 
$(X, E)$, and a mapping $f: X \rightarrow \{-1, 1\}$, for an edge $A \in E$ let 
$f(A):=\sum_{x \in A} f(x)$. The discrepancy of $f$ is $\D(X, E, f)=\max_{A \in E} |f(A)|$, 
while the discrepancy of the hypergraph $(X, E)$

$$\D(X, E):= \min_{f} \D(X, E, f).$$  
  
In our case $X=E(G)$ and $E=\mathcal{S}_G \subset 2^{E(G)}$, and with a slight abuse 
of notation we write $\D(G, \mathcal{S}_G)$ for short. 	
	
Erd\H{o}s, F\"uredi, Loebl, and S\'os \cite{EFLS} studied the case $G=K_n$, the complete graph on $n$ vertices, and 
$\mathcal{S}_G$ is the set of copies of a fixed spanning tree $T_n$ with maximum degree $\Delta$.
They showed the existence of a constant $c>0$, such that $\D(G, \mathcal{S}_G) > c(n-1-\Delta)$. 

 Erd\H{o}s and Goldberg \cite{EGPS} defined ${\rm dis}(A, B):=e(A, B)- e(G)|A||B|/\binom{n}{2}$,
where $A, B \subset V(G)$ and $A \cap B=\varnothing$. They showed that for every $\varepsilon >0$ there exists
an $\varepsilon' >0$ such that in every graph $G$ with $e=e(G) > v(G)=n$, there are disjoint sets $A, B \subset V(G)$,
$|A|, |B| \leq \varepsilon n$, and ${\rm dis}(A, B) > \varepsilon' \sqrt{en}$.

Here we investigate the discrepancy of (spanning) trees, paths and Hamilton cycles. That is for a graph $G$
let $\mathcal{S}_G$ be the set of spanning trees ($\mathcal{T}_n$), trees ($\mathcal{T}$),
Hamiltonian paths ($\mathcal{P}_n$), paths ($\mathcal{P}$), or Hamilton cycles ($\mathcal{H}$). 

Usually, one expects big discrepancy if the hypergraph has 
 many edges. Since for every graph $G$, either $G$ or $\overline G$ is connected, we have
$\D(K_n, \mathcal{T}_n)=n-1$. Beck~ \cite{BeckPath}  showed that there is a graph $F$ on 
$n$ vertices and $2n$ edges such that in every two-coloring of its edge set there exists a monochromatic path 
of length $cn$, that is $\D(F, \mathcal{P}) = cn$. Another example for this is the
interpretation of the result of Burr, Erd\H{o}s and Spencer \cite{BES}, namely that
$R(mK_3, mK_3)=5m$. That is if $k\cdot K_3$ is the set of triangle factors in $K_n$,
$n=3k$ and $n$ is divisible by $5$, then $\D(K_n, k\cdot K_3) = n/5$. 


We first consider the discrepancy of  Hamilton cycles, and show that, roughly speaking, if $G$ has sufficiently large minimum degree then for every labeling of $E(G)$ with $+1,-1$ there is a Hamilton cycle with linear discrepancy.

\begin{theorem}\label{thm:2.1}
Let  $c>0$ be  an arbitrarily small constant and $n$ be sufficiently large. Let $G$ be a graph of order $n$ with $\delta(G)\geq(3/4+c)n$.
Then we have $\D(G,\mathcal{H})\geq cn/32.$ 
\end{theorem}

Figure \ref{fig:3/4} below shows that the minimum degree condition in Theorem~\ref{thm:2.1} is the best possible. In this example, let $G=K_n-K_{n/4}$, i.e., $|V(G)|=n$ is divisible by $4$,  $|V_1|=n/4, |V_2|=3n/4$,  $\delta(G)=3n/4$. Assign $-1$ to all edges incident to $V_1$ and $+1$ to the rest of the edges. As each Hamilton cycle in $G$ touches $V_1$ exactly $n/4$ times, they all have zero discrepancy. 
\begin{figure}[h]
\centering
\includegraphics[width=2.45in]{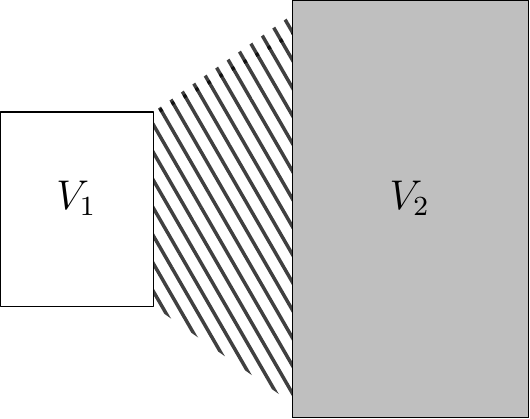}
\caption{$G$ with $\delta(G)=3n/4$ and zero Hamilton cycle discrepancy.}\label{fig:3/4}
\end{figure}

For the existence of a Hamilton cycle,  Dirac's Theorem requires only minimum degree $n/2$. We could also push down the minimum degree requirement for the existence of a linear discrepancy Hamilton cycle, if we have some local restriction on the coloring.

For  $\nu>0$ real number, we say a vertex is $\nu$-\emph{balanced} if it has at least $\nu n$ edges with label $+1$, and at least $\nu n$ edges of label $-1$, otherwise it is $\nu$-\emph{unbalanced}. 
\begin{theorem}\label{thm:main}
Let  $c, d, \nu$ be  positive numbers satisfying  $c\geq8\nu$ and $d\geq 4\nu$. Let $G$ be a graph of order $n$, where $\delta(G)\ge (1/2+c)n$.  Assume that the edges of $G$ are labelled by either $+1$ or $-1$, such that the number of $\nu$-balanced vertices is at least $(3/4+d)n$. Then there exists a Hamilton cycle in $G$ with discrepancy at least $\nu^2 n/500.$
\end{theorem}

The number of the balanced vertices in the graph in Figure~\ref{fig:3/4} is $3n/4$, hence the condition on the size of the balanced set in Theorem \ref{thm:main} is tight.


In both of the theorems above, $G$ is dense. However, the sparsity of a graph does not imply small discrepancy, the expansion is  a more important 
factor. Let $G \in \mathcal{G}_{n, d}$ be a randomly, uniformly selected  $d$-regular graph on $n$ vertices.
A property $\mathcal{P}$ holds with high probability,  w.h.p., if for every $\varepsilon >0$
there exist an $n_\varepsilon$ such that 
${\rm Pr}(G \in \mathcal{G}_{n, d}, G \in \mathcal{P}) \geq 1-\varepsilon$.
Similarly, property $\mathcal{P}$ holds asymptotically almost surely, a.a.s., if 
$\lim_{n \rightarrow \infty} {\rm Pr}(G \in \mathcal{G}_{n, d}, G \in \mathcal{P})=1$.

\begin{theorem}\label{r3reg} 
Let $G \in \mathcal{G}_{n, 3}$. Then there exists a constant $c>0$ such that a.a.s.~we have $\D(G, \mathcal{T}_n) \geq c n$.
\end{theorem} 

For planar graphs, one can expect sublinear discrepancy of spanning trees; we managed to give asymptotically sharp bounds.

\begin{theorem}\label{planar} 
Let $G$ be a planar graph on $n$ vertices. Then there exists a real number $c>0$ such that $\D(G, \mathcal{T}_n) \leq c \sqrt{n}$.
\end{theorem} 

The bounds, up to the constant factor are best possible.  Let $P_k^2:=P_k\stime P_k$ be the 
$k \times k$ grid.

 
\begin{theorem}\label{grid} 
$\D(P_k^2, \mathcal{T}_{n}) \geq ck$ for some $c>0$, where $n=k^2$. 
\end{theorem} 

If we drop the condition of spanning subgraph, then the discrepancies can be linear in the number of vertices.


\begin{proposition}\label{partial} 
Let $k, \ell$ be some positive integers. Then
$\D(P_k\stime P_\ell, \mathcal{P}) > k \ell/8 -\max\{k,\ell\}/8-\min\{k, \ell\}$.
\end{proposition} 

We have the following corollary since paths are also trees.
\begin{corollary}
$\D(P_k\stime P_\ell, \mathcal{T}) > k \ell/8 -\max\{k,\ell\}/8-\min\{k, \ell\}$.
\end{corollary}




Let us make some easy observations which nevertheless give motivations for the above theorems and 
to those proofs.
The graph $P_2 \stime P_k$ has exponentially many spanning trees, but still $\D(P_2 \stime P_k, \mathcal{T}_{2k}) \leq 3$. To see this, we partition the graph into a $2\times \lceil k/2\rceil$ grid and a $2\times \lfloor k/2\rfloor$ grid, and label the edges of the  first grid by $-1$, of the second grid by $+1$. We label the edge shared by two sub-grids arbitrarily. The situation for  $P_k \stime P_k$, the $k\times k$ grid, is similar: 
cut the grid into two halves and label $+ 1$  the upper, and   $-1$ the lower region. Since any spanning tree is cut at most 
$k$ times, $\D(P_k \stime P_k, \mathcal{T}_n) \leq k-1$. For not necessarily spanning trees, obviously, $\D(G, \mathcal{T}) \geq \lceil \Delta(G)/2 \rceil$.

\subsection*{Remark} Koml\'os, S\'ark\" ozy and Szemer\'edi \cite{KSS01} showed that for every $c>0$ and $\Delta$, there is $n_0$ such that if $G$ is a graph of order $n>n_0$ with $\delta(G)>(1/2+c)n$, and $T$ is a tree of order $n$ with maximum degree less than $\Delta$, then $G$ contains $T$ as a subgraph. By using the standard proof method of connected matchings, Theorems~\ref{thm:2.1} and \ref{thm:main} imply the following corollaries.

\begin{corollary}
Let $\Delta$ and  $c> 0$ be given. Then there exists a constant $n_0$ with the following properties. If $n>n_0$, $T$ is a tree of order $n$ with $\Delta(T)\leq \Delta$, and G is a graph of order $n$
with $\delta(G)> (3/4 + c)n$, then there is a subgraph of $G$ which is isomorphic to $T$ with discrepancy $\Theta(n)$.
\end{corollary}

\begin{corollary}
Let $\Delta$ and  $c,d,\nu> 0$ be given. Then there exists a constant $n_0$ with the following properties. If $n>n_0$, $T$ is a tree of order $n$ with $\Delta(T)\leq \Delta$, G is a graph of order $n$ with $\delta(G)> (1/2 + c)n$, and the number of $\nu$-balanced vertices is at least $(3/4+d)n$, then there is a subgraph of $G$ which is isomorphic to $T$ with discrepancy $\Theta(n)$.
\end{corollary}

The key part of the proof is, after applying the degree form of the regularity lemma, find a high discrepancy perfect matching in the cluster graph, which is automatically a connected matching. The proof is standard application of the method of Koml\'os, S\'ark\" ozy and Szemer\'edi \cite{KSS01}, and we omit further details.

\subsection*{Notation} We let $N^+(v)$ to denote the set of neighbors $u$ of $v$ such that $uv$ is labelled by $+1$. Similarly, $N^-(v)$ denotes the set of neighbors $u$ of $v$ such that $uv$ is labelled by $-1$. By definition $N(v)=N^+(v)\cup N^-(v)$. We let $d^+(v)=|N^+(v)|,$ $d^-(v)=|N^-(v)|$ and $d(v)=d^+(v)+d^-(v).$ Suppose $U\subseteq V(G)$, we define $N(U)=\{v\in V(G)\mid \exists u\in U, uv\in E(G)\}$. We say $u$ is a \emph{positive neighbor} of $v$ if $u\in N^+(v)$, and $u$ is a \emph{negative neighbor} of $v$ if $u\in N^-(v)$. Suppose $H$ is a subgraph of $G$, we define $f(H)$ to be 
the sum of labels of all the edges of $H$, where $f:E(G)\to\{1,-1\}$.

\subsection*{Structure of the paper} The paper is organized as follows. In Section 2, we discuss the discrepancy of Hamilton cycles. In Section \ref{sec:ran}, we prove Theorem \ref{r3reg} for random $3$-regular graphs. Section \ref{sec:plan} contains some results of discrepancies for grids and planar graphs.


\section{Discrepancy of Hamilton cycles}\label{sec:ham}

In this section, we study the discrepancy of Hamilton cycles. The first tool we use is the following generalization of Dirac's Theorem \cite{Posa}.
\begin{lemma}\label{lem:posa}
Let $G=(V,E)$ be a graph and let $c>0$ be a real number. Suppose $E^\prime\subseteq E$ such that $E^\prime$ induces a linear forest in $G$. If $\delta(G)\geq(\tfrac{1}{2}+c)n$ and $|E^\prime|\leq 2cn$, then there exists a Hamilton cycle in $G$ which contains all the edges in $E^\prime$.
\end{lemma}

We will use the following simple lemma at various points:

\begin{lemma}\label{lem:path}
Let $\nu,\gamma>0$. Suppose $U\subseteq V(G)$ with $|U|\geq \nu n$ such that for every $u\in U$, we have $|N(u)|\geq \gamma n$. Then there exists a path $P$ of length at least $\nu\gamma n/2$, such that every edge in $P$ contains at least one vertex in $U$. Moreover, if for every $u\in U$ we have $|N(u)\setminus U|\geq \gamma n$, then there exists a path of length at least $\nu\gamma n$ whose vertices are alternating between $U$ and $N(U)\setminus U$.
\end{lemma}

\begin{proof}
Let $H$ be the collection of edges incident to the vertices in $U$. We have $e(H)\geq \nu\gamma n^2/2$. This implies $H$ contains a path $P$ of length at least $\nu\gamma n/2$. It is clear that $P$ satisfies all the requirements. The second part of the statement follows very similarly, considering edges only having exactly one endpoint in $U$.
\end{proof}

Let $G$ be an $n$-vertex simple graph with $\delta(G)\ge (3/4+c)n,$ where $c>0$ is a (possibly small) constant and the edges of $G$ are labelled by either $+1$ or $-1$. 

\begin{proof}[Proof of Theorem \ref{thm:2.1}]
Let $a=c/4$. The proof is split into two cases.
\medskip

\Case{1}{At least $(3/4+c)n$ vertices in $G$ are $a$-balanced.}

Suppose there exists a vertex $v$ such that less than $cn/2$ vertices in $N(v)$ have more than $cn$ negative neighbors in $N(v)$. Let $M\subseteq N(v)$ be the set of such vertices, hence, $|M|<cn/2$. Note that each vertex in $N(v)$ has at least $(1/2+2c)n$ neighbors inside $N(v)$, hence $G[N(v)\setminus M]$ contains a Hamilton cycle $\C$ with all edges being positive. Then we insert those vertices not in $N(v)\setminus M$ one by one to $\C$, so we obtain a Hamilton cycle with discrepancy at least
\[
|N(v)\setminus M|-3|V\setminus \big(N(v)\setminus M\big)|\geq 2cn.
\]

Now suppose that such vertex does not exist. 
Let $S\subseteq V(G)$ be the set of balanced vertices which have more positive neighbors. We may assume 
$|S|\geq(3/4+c)n/2$. Then for every $v\in S$, each vertex in $N^+(v)$ has at least $(1/8+3c/2)n$ neighbors in $N^+(v)$. Every vertex of $S$ has at least $a$ negative neighbors, hence, using Lemma~\ref{lem:path} we get that there exists 
a negative path $P$, such that each edge of $P$ contains at least one vertex in $S$, and both of the end vertices of $P$ are in $S$. Moreover, the length of $P$ is at least $\tfrac{an}{2}(\tfrac{3}{4}+c).$ We denote the end vertices of $P$ by $x,y$.

Next for each vertex $v$ in $V(P)\cap S$, we pick an edge in $N^+(v)$. For each vertex in $P$ but not in $S$, we pick a negative edge from its neighborhood. We also pick an edge $ab$ such that $a\sim x$ and $b\sim y$. We require that all the edges we picked are disjoint from $P$ and they form a linear forest in $G$. This is doable, since in each step we forbid edges incident to the vertices in $V'\subseteq V(G)$ with $|V'|<cn/2$.

Let $G'$ be the graph after we delete $P$ from $G$. By Lemma~\ref{lem:posa}, there is a Hamilton cycle $\C$ in $G'$ containing all the edges we picked.
First, we insert the entire path $P$ by removing the $ab$ edge and adding edges $ax$ and $by$. We obtain a Hamilton cycle $\C_1$, such that
\[
f(\C_1)\leq f(\C)-|P|+3\leq f(\C)-\frac{an}{2}\Big(\frac{3}{4}+c\Big)+3.
\]  
If $f(\C)\le 3cn/64,$ then the above implies that $f(\C_1)\le -3cn/64.$ If $f(\C)>3cn/64,$ then we can 
insert the vertices in $P$ one by one to obtain $\C_2$, such that we have
\[
f(\C_2)\geq f(\C)+\frac{|P|}{2}-\frac{|P|}{2}=f(\C).
\]
Therefore, $G$ contains a Hamilton cycle with discrepancy at least $3cn/64$.
\medskip

\Case{2}{There are at least $2cn$ vertices in $G$ which are not $a$-balanced.}

Suppose there exists a set $T$ containing $cn$ unbalanced vertices, each having at most $an$ negative neighbors. Let $\C$ be a Hamilton cycle in $G$. The difference between the number of positive edges and negative edges of $\C$ is at most $cn$, otherwise we are done. Then for every vertex $v\in T$, $N^+(v)$ contains at least $(2c-2a-\tfrac{c}{2})n$ positive edges of $\C$ and at least $(2c-2a-\tfrac{c}{2})n$ negative edges of $\C$. For each vertex $v$ in $T$, we pick a positive edge and a negative edge in $N^+(v)\cap \C$. 

Now we define $G':=G-T$. By Lemma~\ref{lem:posa}, there is a Hamilton cycle $\C$ in $G'$ that contains all the edges we picked.
We can either remove all the negative edges we picked in $\C'$ to insert the vertices in $T$, or remove all the positive edges we picked. Clearly, $G$ contains a Hamilton cycle with discrepancy at least $cn$.
\end{proof} 


Now we need some preparation to prove Theorem \ref{thm:main}. Let $\T$ be the set of triangles in $G$. We define a function $$g: V(G)\times \T\to\{-3,-1,0,1,3\},$$ such that for every $v\in V(G)$ and triangle $T\in\T$, we let $g(v,T)=0$ if $v$ is not a vertex in $T$. For $T=uvw$, we let $g(v,T)$ be the change in the discrepancy if the edge $uw$ is changed to the path $uvw$. To be more precise, we let $g(v,T)$ be $-1,1,-3,3$ if the triangle $T$ has type red, blue, dark red, dark blue, respectively, see Figure \ref{fig:1}.

\begin{figure}[h]
\centering
\begin{minipage}[t]{0.45\textwidth}
\centering
\includegraphics[width=2.2in]{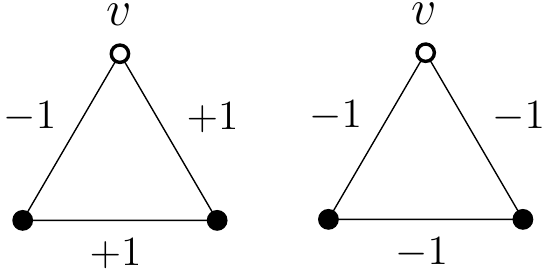}
\captionsetup{labelformat=empty}
\caption*{Type red.}
\end{minipage}\hfill\begin{minipage}[t]{0.5\textwidth}
\centering
\includegraphics[width=2.2in]{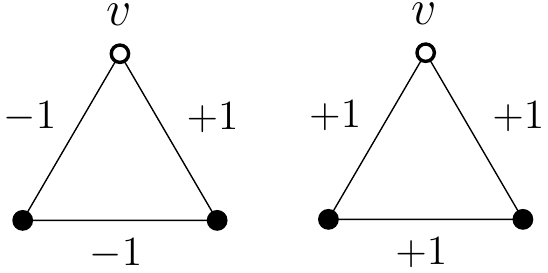}
\captionsetup{labelformat=empty}
\caption*{Type blue.}
\end{minipage}
\end{figure}
\begin{figure}[h]
\centering
\begin{minipage}[t]{0.5\textwidth}
\centering
\includegraphics[width=1.0in]{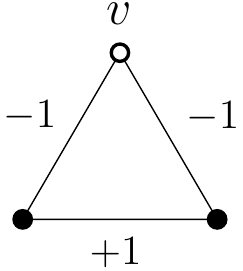}
\captionsetup{labelformat=empty}
\caption*{Type dark red.}
\end{minipage}\hfill\begin{minipage}[t]{0.5\textwidth}
\centering
\includegraphics[width=1.0in]{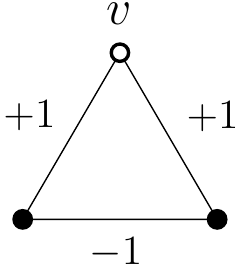}
\captionsetup{labelformat=empty}
\caption*{Type dark blue.}
\end{minipage}
\caption{The vertex coloring of $G$.}\label{fig:1}
\end{figure}

We color the vertex $v$ \emph{red} if there exist at least $\nu n^2$ triangles $T$ in $\T$ such that $g(v,T)=-1$, it is \emph{blue, dark red, dark blue} if there exist at least $\nu n^2$ triangles $T$ in $\T$ such that $g(v,T)=1,-3,3$, respectively. Note that when $c>8\nu$, every vertex is colored, since the neighborhood of every vertex spans at least $cn^2/2$ edges. Some vertices may have multiple colors under this definition, but we may assume most of them have only one color, using the
following lemma.
\begin{lemma}\label{lem:mono}
Suppose more than $\nu n/3$ vertices have more than one colors. Then there exists a Hamilton cycle of discrepancy at least $\nu n/3$.
\end{lemma}
\begin{proof}
Let $M\subseteq V(G)$ be the set of vertices having more than one colors, and $x_1,x_2,\dots,\allowbreak x_{\nu n/3}\in M$. For every $x_i$ ($1\leq i\leq \nu n/3$), we pick edges $a_ib_i$ and $c_id_i$ in the neighborhood of $x_i$, such that $g(x_i,x_ia_ib_i)\neq g(x_i,x_ic_id_i)$. We further require that, all the edges we picked do not contain $x_i$, and they form a linear forest in $G$. We can do this, since in each step we forbid less than $\nu n^2$ edges, but we have at least $\nu n^2$ triangles by the definition.

Now we remove $x_1,x_2,\dots,x_{\nu n/3}$ from $G$, and call the resulted graph $G'$. By Lemma~\ref{lem:posa}, we can find a Hamilton cycle $\C$ in $G'$ containing all the edges we picked. In order to insert $x_i$ back to $\C$, we can remove either $a_ib_i$ or $c_id_i$, and in each step, the discrepancies differ by at least $2$, since $|g(x_i, x_ia_ib_i)-g(x_i, x_ic_id_i)|\ge 2.$ Therefore, there exists a Hamilton cycle in $G$ with discrepancy at least $\nu n/3$.
\end{proof}


The following Lemma is our main tool in the proof.
\begin{lemma}\label{lem:tool}
Let $c,\nu>0$ with $c>8\nu$. Let $G$ be a graph with $\delta(G)\geq (1/2+c)n$. Let $R,Q\subseteq \{\text{red}, \text{blue}, \text{dark red}, \text{dark blue}\}$.  
Suppose, there is a path $P$ of length $\phi(\nu)n$ for some function $\phi$ and all edges of it have labels in $I\subseteq\{+1,-1\}$, where $0<\phi(\nu)<\nu/2$, and each edge of $P$ contains at least one vertex with colors in $R$, and the other vertices on $P$ have colors in $Q$. Assume that one of the following holds:\medskip\\
\emph{(i)} $\ \,I=\{-1\}$, $R=\{ \text{dark blue}\}$.\\
\emph{(ii)} $\: I=\{-1\}$, $R=\{ \text{blue}\}$, dark red $\notin Q$.\\
\emph{(iii)} $R=Q=\{ \text{dark blue}\}$.\\
\emph{(iv)} $R=\{ \text{dark blue}\}$, $Q=\{\text{blue}\}$.\\
\emph{(v)} $\;I=\{+1\}$, $R=Q=\{ \text{red}\}$.\medskip\\
Then if one of \emph{(i)}, \emph{(ii)}, \emph{(v)} holds, $G$ contains a Hamilton cycle with discrepancy at least $\phi(\nu)n/2-3/2$. If one of  \emph{(iii)}, \emph{(iv)} holds, $G$ contains a Hamilton cycle with discrepancy at least $\phi(\nu)n/4-3/4$. 
\end{lemma}
\begin{proof}
Let $X$ be the set of vertices on $P$ with colors in $R$, and let $Y$ be the set of vertices with colors in $Q$. Suppose $x,y$ are the first and the last vertices in $P$. 

Let us focus on (i) first. For every vertex $v\in X$, we pick an edge $a_vb_v$ inside the neighborhood of $v$, such that $a_v,b_v\notin V(P)$, and $g(v,va_vb_v)=3$. We require that the edges we picked form a linear forest. This is possible, and we can pick the edges one by one. For each step, the edge we chose cannot contain a vertex which already used twice in the previously chosen edges, and two end vertices of the new edge cannot both already used. Clearly, the number of edges that cannot be chosen is strictly less than $\nu n^2$, but we have $\nu n^2$ options, by the definition of the dark blue vertices.

For every vertex $u$ in $Y$, we pick an edge $a_ub_u$ in $N(u)$, and we pick the edge $ab$ such that $a\sim x$ and $b\sim y$, so for the endpoints $x$ and $y$ we pick two edges. Together with the edges we picked for the vertices in $X$, we further require that all the edges we picked are disjoint from $P$ and they form a linear forest in $G$. Note that the number of edges we picked is less than $cn$.

Let $G'$ be the graph after removing all the vertices in $P$ from $G$, we have $\delta(G')\geq (1/2+c/2)n$. Now we apply Lemma \ref{lem:posa}, and suppose $\C$ is a Hamilton cycle in $G'$ containing all the edges we picked. We have two different ways to construct a Hamilton cycle in $G$.

We remove the edge $ab$ and add $ax,by$ to insert the entire path $P$, we denote the resulted Hamilton cycle by $\C_1$. Clearly, 
\[
f(\C_1)\leq f(\C)-|P|+3=f(\C)-\phi(\nu) n+3.
\]

We can also insert the vertices in $P$ one by one. That is, for every $v\in V(P)$, we remove the edge $a_vb_v$ in $\C$ and add the edges $va_v, vb_v$. We then obtain a Hamilton cycle $\C_2$, and we have
\[
f(\C_2)\geq f(\C)+3|X|-3|Y|=f(\C),
\]
since the worst case is when all the vertices in $Y$ are dark red. Therefore, we obtain a Hamilton cycle in $G$ with discrepancy at least $\tfrac{1}{2}(\phi(\nu)n-3)$.
 
Now we consider (ii). The ideas are similar: For every vertex $v$ in $X$, we pick an edge $a_vb_v$ in $N(v)$ such that the $g(v,va_vb_v)=1$. For every vertex $u$ in $Y$, we pick an edge $a_ub_u$ in $N(u)$ such that $g(u,ua_ub_u)\neq -3$. We also pick $ab$ adjacent to the end vertices of $P$, and we require all the edges we picked are disjoint from $P$, and they form a linear forest.

We now remove all the vertices in $P$ from $G$. Let $\C$ be the Hamilton cycle in the resulted graph which contains all the edges we picked.
We can either insert the entire path to $\C$, or insert the vertices one by one. In the second situation, the worst case is when all the vertices in $Y$ are red. This gives us a Hamilton cycle with discrepancy at least $\tfrac{1}{2}(\phi(\nu)n-3)$.

Note that (ii) implies (v), since we can map $-1$ to $+1$, blue to red, and dark blue to dark red. For cases (iii) and (iv), for vertices in $X$, we pick edges as we did in (i). For the vertex $u$ in $Y$, we pick $a_ub_u$ in $N(u)$ such that $g(u, ua_ub_u)$ is $3$ and $1$, respectively. Again we have two ways to obtain the Hamilton cycle in $G$, insert the entire path, or insert the vertices one by one. Note that the worst case is when all the edges in $P$ are labelled by $1$. But since we have at least half of the vertices in $P$ dark blue, the difference of the discrepancies between these two constructions is still large, and we obtain a Hamilton cycle with discrepancy at least $\phi(\nu)n/4-3/4$. We omit further details.
\end{proof}

\noindent{\bf Remark.} Note that if we reverse the colors and the labels simultaneously, the same conclusions in Lemma~\ref{lem:tool} still hold.
\medskip


With all tools in hand, we are going to prove Theorem \ref{thm:main}.
\begin{proof}[Proof of Theorem \ref{thm:main}]
Let $M\subseteq V(G)$ be the set of vertices having more than $1$ colors. By Lemma \ref{lem:mono}, we have $|M|<\nu n/3$. Let $A,B,C,D\subseteq V(G)\setminus M$ be the set of blue, red, dark blue and dark red vertices, respectively. By Lemmas \ref{lem:path} and \ref{lem:tool}, we may assume the following properties of $G$.
\medskip

(i) At most $\nu n/30$ vertices in $C$ ($D$) have more than $\nu n/4$ neighbors in $C$ ($D$).
Otherwise by Lemma \ref{lem:path} we can find a path $P$ of length $\nu^2 n/120$ either inside $C$, or inside $D$. In both cases, condition (iii) in Lemma \ref{lem:tool} gives us a Hamilton cycle of discrepancy at least $\nu^2 n/480-3/4$. 
\smallskip

(ii) At most $\nu n/30$ vertices in $C$ ($D$) have more than $\nu n/4$ neighbors in $A$ ($B$). If not, there is a path of length $\nu^2 n/120$ whose vertices alternate between $C$ and $A$ (between $D$ and $B$), and condition (iv) in Lemma \ref{lem:tool} gives us a Hamilton cycle of discrepancy at least $\nu^2 n/480-3/4$.
\smallskip

(iii) At most $\nu n/3$ vertices in $A$ ($B$) have more than $\nu n/6$ negative (positive) neighbors inside $A$ ($B$). By the same reason as above, otherwise condition (v) in Lemma~\ref{lem:tool} gives us a Hamilton cycle of discrepancy at least $\nu^2 n/36-3/2$.
\smallskip

(iv) At most $\nu n/30$ vertices in $C$ ($D$) have more than $\nu n/4$ neighbors in $D$ ($C$). If, say, at least $\nu n/30$ vertices in $C$ have more than $\nu n/4$ neighbors in $D$, then suppose $\nu n/60$ of them have more positive neighbors in $D$. By Lemma \ref{lem:path}, there is a positive path $P$ of length $\nu^2 n/240$ whose vertices alternate between $C$ and $D$. We now apply condition (i) in Lemma \ref{lem:tool}, but in the form that $I=\{+1\}$ and $R=\{\text{dark red}\}$. Thus there exists a Hamilton cycle with discrepancy at least $\nu^2n/480-3/2$.
\smallskip

(v)  At most $\nu n/3$ vertices in $A$ ($B$) have more than $\nu n/6$ neighbors in $B$ ($A$). By the same reason as above, if we have more than $\nu n/3$ vertices in $A$ having more than $\nu n/6$ neighbors in $B$, we may suppose that $\nu n/6$ of them have more positive neighbors in $B$. Thus by Lemma~\ref{lem:path} there is a path of length $\nu^2 n/72$ whose vertices alternate between $A$ and $B$.
  Then we apply condition (ii) in Lemma \ref{lem:tool} with $I=\{+1\}$ and $R=\{\text{red}\}$, there exists a Hamilton cycle with discrepancy at least $\nu^2n/144-3/2$.
\smallskip

(vi) At most $\nu n/30$ vertices in $C$ ($D$) have more than $\nu n/4$ negative (positive) neighbors in $B$ ($A$). Otherwise condition (i) in Lemma \ref{lem:tool} gives a Hamilton cycle of discrepancy at least $\nu^2 n/240-3/2$.
\smallskip

(vii) At most $\nu n/3$ vertices in $A$ ($B$) have more than $\nu n/6$ neighbors in $C$ ($D$). If not, the condition (iv) in Lemma \ref{lem:tool} gives a Hamilton cycle of discrepancy at least $\nu^2 n/72-3/4$.
\smallskip

(viii) At most $\nu n/3$ vertices in $A$ ($B$) have more than $\nu n/6$ positive (negative) neighbors in $D$ ($C$). If not, the condition (i) in Lemma \ref{lem:tool} gives a Hamilton cycle of discrepancy at least $\nu^2 n/36-3/2$.
\medskip

Now the approximate structure of $G$ is as follows. The graph induced on $C\cup D$ is almost empty, and $G[A,C]$, $G[B,D]$, $G[A,B]$ are almost empty. Almost all the edges between $A$ and $D$ are negative, and almost all the edges between $B$ and $C$ are positive. Almost all the edges inside $A$ are positive, and almost all the edges inside $B$ are negative. 

We say a vertex in a set is \emph{typical} if it behaves as almost all the vertices in this set, otherwise it is \emph{untypical}. More precisely, a vertex $v\in A$ ($B$) is typical if it has less than $\nu n/6$ negative (positive) neighbors in $A$ ($B$), less than $\nu n/6$ neighbors in $B$ ($A$), less than $\nu n/6$ neighbors in $C$ ($D$), and less than $\nu n/6$ positive neighbors in $D$ ($C$). A vertex $v\in C$ ($D$) is typical if it has less than $\nu n/4$ neighbors in $C$ ($D$), less than $\nu n/4$ neighbors in $A$ ($B$), less than $\nu n/4$ neighbors in $D$ ($C$), and less than $\nu n/4$ negative (positive) neighbors in $B$ ($A$).

The rest of the proof is based on analyzing the number of dark vertices.

\Case{1}{There exist at most $\nu n/6$ dark blue vertices and at most $\nu n/6$ dark red vertices.}

In this case, we have $|A\cup B|\geq (1-\tfrac{\nu}{3})n$. Suppose $|A|\geq (\tfrac{1}{2}-\tfrac{\nu}{6})n$, and let $A'\subseteq A$ be the set of $\nu$-balanced vertices. Clearly, $|A'|\geq (\tfrac{1}{4}+d-\tfrac{\nu}{6})n$, and  each vertex $v\in A'$ has at least $\nu n$ negative neighbors. By (iii), (v) and (vii), all but $\nu n$ vertices in $A'$ have less than $\nu n/6$ negative neighbors inside $A$, in $B$ and in $C$. Since $|D|\leq \nu n/6$ and $|M|<\nu n/3$, we get a contradiction.

\Case{2}{There are at least $\nu n/6$ dark blue vertices or $\nu n/6$ dark red vertices.}

Suppose $|C|\geq \nu n/6$. By (i), (ii), (iv), and (vi), we have that at most $2\nu n/15$ vertices in $C$ are untypical, which implies that all the other vertices in $C$ are $\nu$-unbalanced. Hence $|C|\leq (\tfrac{1}{4}-d+\tfrac{2\nu}{15})n$, and $|B|\geq(\tfrac{1}{2}+c-\frac{3}{4}\nu)n$, since the typical vertices in $C$ have at most $3\nu/4$ vertices outside of $B$. This also gives us $|D|\leq \nu n/6$, otherwise we would also have $|A|\geq(\tfrac{1}{2}+c-\frac{3\nu}{4})n$, and this contradicts with $A\cap B=\varnothing$.

Thus, we have $|A|\leq (\tfrac{1}{2}-c+\tfrac{7\nu}{12})n$, and actually this implies $|A|\leq 4\nu n/3$. The reason for this is, first by (iii), (v), (vii), and (viii), $A$ contains at most $4\nu n/3$ untypical vertices. By (v) and (vii), the typical vertices in $A$ have all but at most $\nu n/3$ of their neighbors in $A$ and $D$. But $|A\cup D|\leq (\tfrac{1}{2}-c+\tfrac{3\nu}{4})n$, which means that all the vertices in $A$ are untypical because of the disjointness of $A,B,C,D,M$. Therefore, we have $|A\cup C\cup D\cup M|\leq (\frac{1}{4}-d+2\nu)n$, and thus $|B|\geq(\frac{3}{4}+d-2\nu)n$.

Let $B'\subseteq B$ be the set of typical vertices in $B$, and let $C'\subseteq C$ be the set of typical vertices in $C$. Clearly, we have $|B\setminus B'|\leq \nu n$ and $|C\setminus C'|\leq \frac{2}{15}\nu n$. Let $K$ be a graph such that $V(K)=B'\cup C'$, and $e\in E(K)$ if $e$ is either a negative edge in $G[B']$ or a positive edge in $G[B',C']$. 
Now, $|V(K)|\geq (1-3\nu)n$, since besides $B\setminus B'$ and $C\setminus C'$, we have $|A|\leq 4\nu n/3$, $|D|\leq \nu n/6$, and $|M|\leq \nu n/3$. Also, for every $v\in V(K)$, $\delta_K(v)\geq(\tfrac{1}{2}+c-3\nu)n$. This is because, for every $u\in B'$, by (v) and the size of $D$ and $M$, all but at most $2\nu n/3$ neighbors of $u$ are in $B\cup C$. By (iii) and (viii), $u$ has at most $\nu n/6$ positive neighbors in $B$ and at most $\nu n/6$ negative neighbors in $C$. By the size of $B\setminus B'$ and $C\setminus C'$, we have $\delta_K(u)\geq (\tfrac{1}{2}+c)n-2\nu n-\tfrac{2\nu n}{15}$. Similarly, for every $w\in C'$, by (i), (ii), (iv), (vi), and the size of $M$ and $B\setminus B'$, $\delta_K(w)\leq (\tfrac{1}{2}+c)n-2\nu n-\tfrac{\nu n}{3}$. Therefore, $K$ contains a Hamilton cycle $\C$. Since $C'$ is an independent set in $K$ by (i),  the number of positive edges in $\C$ is at most $2|C'|\leq2|C|\leq(\frac{1}{2}-2d+\nu)n$.

Now we go back to $G$. Note that $\C$ is also a Hamilton cycle in $G[B'\cup C']$. In the final step, we are going to insert all the vertices in $V(G)\setminus (B'\cup C')$ to $\C$. Let $J=(B\setminus B')\cup (C\setminus C')\cup A\cup D\cup M$. We have $|J|\leq 3\nu n$. Then after we insert all vertices in $J$ to $\C$, we obtain a Hamilton cycle in $G$, which contains at most $(\frac{1}{2}-2d+\nu)n+2|J|=(\frac{1}{2}-2d+7\nu)n$ positive edges. Therefore, $G$ contains a Hamilton cycle with discrepancy at least $(4d-14\nu)n>2\nu n$.
\end{proof}


\section{Discrepancies in random $3$-regular graphs}\label{sec:ran}
\begin{proof}[Proof of Theorem~\ref{r3reg}]


Buser \cite{Buser} and later, in a much simpler paper, Bollob\'as \cite{BBregular} showed that random 
regular graphs have expanding properties. More precisely, let $$i(G):= \min_{\:U} \frac{|\partial U|}{|U|},$$
where $U \subset V(G)$ with $|U| \leq |V(G)|/2$, and $\partial U:=\{v\notin U\mid \exists u\in U, uv\in E(G)\}$. 
\smallskip

(i) Bollob\'as \cite{BBregular} proved that $i(G) \geq 2^{-7}$ for a random 3-regular graph $G$ with high probability. 
In particular, it is connected w.h.p..
\smallskip

(ii) Bollob\'as \cite{BB2} showed  for 
$3 \leq j \leq k$, where $k$ is fixed, and $X_j$ stands for the number of cycles of length $j$ in 
$G \in \mathcal{G}_{n, 3}$, that $X_3, \dots, X_k$ are asymptotically independent Poisson random 
variables with means $\lambda_j=2^j/(2j)$.

\smallskip

(iii)  Wormald proved 
(see \cite[Lemma 2.7]{W}) that for a fixed $d$ and every fixed graph $F$ with more edges than vertices,
$G \in \mathcal{G}_{n, d}$ a.a.s.~contains no subgraph isomorphic to $F$.

\medskip
 
Fix an arbitrary $f:E(G) \rightarrow \{-1, 1\}$,  denote $N$ and $P$  the subsets of {\em edges}, 
where $f$ takes $-1$ and $1$, respectively. We may assume that $|N| \leq |P|$, i.e., $|N|\le 3n/4$. 

Denote by $G^+$ the subgraph of $G$ spanned by $P$, and let $A_i$ be the set of components with 
size $i$ in $G^+$, while $a_i:=|A_i|$. The number of components in $G^+$ is $t=\sum_{i=1}^n a_i$.

Note that (i) means that $G$ is connected w.h.p.~so $G$ has a spanning tree $T$ satisfying that
$|E(T) \cap N| \leq t-1$.
Hence if  $t \leq (1/2-2^{-12})n+o(n)$ or $t \geq (1/2+2^{-12})n+o(n)$ then  $\D(G, \mathcal{T}_n) \geq 2^{-12}n-o(n)$.

Three edges of $N$ are incident to each element of $A_1$,  four edges to 
each of $A_2$.
 The number of edges incident to a component of size at least $3$   could  be less than 
four  only if the component contains a cycle, i.e.,  w.h.p.~only in $O(1)$ many components $A_i$ for  $i=3, \dots, 2^9$. For every component
larger than $2^9$, and smaller than $n/2$,  w.h.p. the number of incident edges is  at least four by (i). 

That is, w.h.p. $$2|N| \geq 3a_1 + 4\sum_{i=2}^n a_i -O(1)=4t-a_1-O(1),$$which gives 
\begin{equation}\label{eq:1}
(1/2-2^{-12})n+O(1) \leq t \leq |N|/2 +a_1/4 +O(1) \leq 3n/8 +a_1/4 +O(1).
\end{equation}

Now we consider the number of negative edges. The number of edges in $N$ which are incident to vertices in $A_1$ is
$e(G[A_1])+e(G[A_1,\overline{A}_1])$. Since $|N| \leq |P|$, we have $$\frac{3a_1}{2}\leq e(G[A_1])+e(G[A_1,\overline{A}_1])\leq|N|\leq\frac{3n}{4},$$ which implies that $a_1\leq n/2$. Using the condition (i), we have $e(G[A_1,\overline{A}_1])\geq 2^{-7}a_1$. Therefore,
\[
3a_1\leq2e(G[A_1])+e(G[A_1,\overline{A}_1])\leq 2|N|-2^{-7}a_1,
\]
implying $$\frac{3a_1}{2} +\frac{a_1}{2^7} \leq |N| \leq \frac{3n}{4}, $$ which gives  
$a_1 \leq (1/2 -2^{-10})n$ w.h.p. With (\ref{eq:1}) it implies $t \leq (1/2-2^{-12})n+o(n)$~w.h.p.
That gives us $\D(G, \mathcal{T}_n) \geq 2^{-12}n-o(n)$. w.h.p. 
\end{proof}

\section{Discrepancies of planar graphs}\label{sec:plan}

\begin{lemma}\label{cut} Let $C$ be a vertex cut of a connected graph $G$, that is $V(G)=A \cup B \cup C$ such that there 
are no edges between $A$ and $B$, and, say, $|A| \leq |B|$.  Then $\D(G, \mathcal{T}_n) \leq |B|-|A|+|C|$.
\end{lemma}

\begin{proof}
Let $f(x, y)=1$ if $(x, y) \in E(A) \cup E(A, C)$, $f(x, y)=-1$ if $(x, y) \in E(B) \cup E(B, C)$ and 
arbitrary in $E(C)$. Every spanning tree $T$ of $G$ has at most $|C|$ components restricted to $A \cup C$.
It means the number of edges labeled by $1$ is at least $|A|+|C|-1-|C|=|A|-1$ in $T$, and the edges labeled
by $-1$ at most $|B|+|C|-1$. 
\end{proof}

\begin{proof}[Proof of Theorem~\ref{planar}]

To deduce Theorem~\ref{planar} we need to recall the celebrated planar separation theorem of Lipton and
Tarjan in \cite{LT}. It says if $G$ is a planar graph on $n$ vertices then $G$ has a vertex cut of size 
$O(\sqrt{n})$ partitioning the graph into two parts $A$ and $B$, where $n/3 \leq |A|, |B| \leq 2n/3$. A well-known consequence \cite[Theorem~5]{D82} of that theorem is that there
exists a cut $C$ and constants $c_1,c_2,c_3$ such that $n/2 - c_1\sqrt{n} \leq |A|, |B| \leq n/2 + c_2\sqrt{n}$ and $|C|=c_3\sqrt{n}$.

\smallskip       

Having the partition above we can use Lemma~\ref{cut} getting that for a planar graph $G$, 
$\D(G, \mathcal{T}_n) \leq |B|-|A|+|C| \leq O(\sqrt{n})$. 
\end{proof}

\begin{lemma}[\cite{Chvatalova}]\label{lem:Chvat}
Let $S\subseteq P_k \stime P_k$ such that $(k^2-k)/2\leq |S|\leq (k^2+k)/2$. Then we have $|\partial S|\geq k$.
\end{lemma}

\begin{proof}[Proof of Theorem~\ref{grid}]
Assume there exists an $f:E(P_k \stime P_k) \rightarrow \{-1, 1\}$ such that $\D(P_k \stime P_k, \mathcal{T}_n, f) \leq k/4$.
Let $P, N$ and $M$ be the subset of vertices, such that $v\in P$ if all edges incident to $v$ are positive, $v\in N$ if all edges incident to $v$ are negative, and $M=V-N-P$.
Consider an arbitrary Hamiltonian path in $P_k \stime P_k$, from the assumption on $f$ it follows that $|P|, |N| \leq k^2/2 + k/8 +2$.

First, we show that $|M| \geq k $. If $\max \{|P|, |N|\} \leq (k^2-k)/2$ then this follows from 
$|P|+|N|+|M| =k^2$. That is we may assume $(k^2-k)/2 < |P| \leq k^2/2 + k/8 +2$. Note that $\partial P=M$. By Lemma~\ref{lem:Chvat}, for sets $P$ of such size we have $|\partial P|\geq k$, which means $|M| \geq k$, too.

We identify the vertices of $P_k \stime P_k$ with coordinate pairs such that $(0, 0)$ belongs to the bottom left vertex,
$(k-1, k-1)$ to the upper right vertex. For $r, s \in \{0, 1\}$ let $X_{r, s}$ be those vertices $(i, j)$
($0\leq i, j \leq k-1$) for which $i= r\pmod 2$ and $j = s\pmod 2$. At least one of these sets $X_{r, s}$ contains at least
$k/4$ vertices of $M$, say $X_{0, 0}$. Consider an arbitrary tree $T$ spanned on the vertices $X_{0, 1} \cup X_{1, 0} \cup X_{1, 1}$.

Note that we can extend $T$ to the entire $P_k \stime P_k$ such that the vertices of $X_{0, 0}$ will be leaf 
vertices in the extension. Moreover for $(i, j) \in X_{0, 0} \cap M$ we can connect $(i, j)$ to $T$ with either
an edge labeled by $-1$ or $1$. Fixing any extension to $X_{0, 0} \setminus M$, let $T^+$ ($T^-$) be the extension where
we use the edge labeled by $1$ ($-1$) for the vertices $ X_{0, 0} \cap M$. Obviously,  
$|\sum_{e \in T^+} f(e) - \sum_{e \in T^-} f(e)| \geq k/2$, so either $|\sum_{e \in T^+} f(e)|$ or 
$|\sum_{e \in T^-} f(e)|$ is at least $k/4$.   
\end{proof}

\begin{proof}[Proof of Proposition~\ref{partial}]

We show first that $\D(P_k \stime P_2, \mathcal {P}) \geq k/2$. Let us refer to the graph $P_k \stime P_2$
as a rectangle with horizontal length $k$ in which the edges are labeled by $f$. Let $X$ and $Y$ be the set of the vertical edges
labeled by $+1$ and $-1$ respectively. Without loss of generality, we may assume $|X| \geq |Y|$ and let $x:=|X| \geq k/2$, $y:=|Y|$. We consider four paths: $P(X)$ starts from the left-upper corner 
goes to right except when it meets an edge $e \in X$ at which point it goes down or up, depending on which one is 
possible. The path $P'(X)$ is almost the same, but it starts from the left-lower corner. 
Finally the paths $P(Y)$ and $P'(Y)$ are drawn analogously, those also start from left and go to right, but rise and
fall at the edges belonging to $Y$. Note that $P(X)$ and $P'(X)$ each contain $X$, $P(Y)$ and $P'(Y)$ each contain $Y$. $P(X)\cup P'(X)$ and $P(Y)\cup P'(Y)$ have the same set of horizontal edges.

Let $z_1:=\sum_{e \in P(X) \setminus X} f(e)$, and
 $z_2:=\sum_{e \in P'(X) \setminus X} f(e)$. If $\max \{z_1, z_2\} \geq 0$, then we are done since 
one of $\sum_{e \in P(X)} f(e)$ or $\sum_{e \in P'(X)} f(e)$ is at least $k/2$. If both $z_1$ and $z_2$ are
negative, we have 
$\D(P_k \stime P_2, \mathcal {P}, f) \geq x+z_1$, and $\D(P_k \stime P_2, \mathcal {P}, f) \geq x+z_2$. 
Considering the paths $P(Y)$ and $P'(Y)$ we also have $2\D(P_k \stime P_2, \mathcal {P}, f) \geq 2y-z_1-z_2$,
since the horizontal edges in those carry exactly $z_1+z_2$ negative surplus. Adding those up, we get
$4\D(P_k \stime P_2, \mathcal {P}, f) \geq 2x+2y$, that is 
$\D(P_k \stime P_2, \mathcal {P}, f) \geq k/2$ since $x+y=k$.

In the general case we may assume that $k \leq \ell$ and $P_k \stime P_\ell$ is referred as a rectangle
with $k$ rows and $\ell$ columns. We cut out $\lfloor k/2\rfloor$ non-touching stripes $P_2\stime P_\ell$. For every $f: E(P_k \stime P_\ell) \rightarrow \{-1, 1\}$, applying our construction of paths above, without loss of generality, at least half of the rectangles have a path with more positive edges, and with discrepancy at least $\lceil \ell/2\rceil$. Note also, that these paths can be joined into one path by adding at most $k-1$
edges. Thus, we create a path with discrepancy at least
\[
\bigg\lceil\frac{1}{2}\Big\lfloor \frac{k}{2}\Big\rfloor\bigg\rceil\Big\lceil \frac{\ell}{2}\Big\rceil-k+1> \frac{k\ell}{8}-\frac{\ell}{8}-k,
\]
and the result is proved.
\end{proof}

\noindent{\bf Remark.} Motivated by Theorem~\ref{thm:2.1} and $\D(K_n, k\cdot K_3) = n/5$ from \cite{BES} we think 
that for any $c>0$,  $\D(G, k\cdot K_3) = \Theta(n)$ provided that $v(G)=n$ and $\delta(G) \geq (3/4 + c)n.$

\end{document}